\documentclass[11pt,a4paper]{amsart}
\usepackage[T1]{fontenc}
\usepackage[utf8]{inputenc}
\usepackage{lmodern,amsmath,amssymb,amsthm,mathtools,mathrsfs}
\usepackage[margin=28mm]{geometry}
\usepackage[expansion=false]{microtype}
\usepackage{enumitem,booktabs}
\usepackage{xcolor}
\definecolor{linkblue}{RGB}{28,54,84}
\usepackage[colorlinks=true,linkcolor=linkblue,citecolor=linkblue,urlcolor=linkblue]{hyperref}
\hypersetup{pdftitle={Rigidity, sharp inequalities, and stability for sigma-two curvature},pdfauthor={Wangzhe Wu}}
\numberwithin{equation}{section}
\newtheorem{theorem}{Theorem}[section]
\newtheorem{proposition}[theorem]{Proposition}
\newtheorem{lemma}[theorem]{Lemma}

\newtheorem{maintheorem}{Theorem}

\theoremstyle{definition}

\theoremstyle{remark}
\newtheorem{remark}[theorem]{Remark}
\newcommand{\Ric}{\operatorname{Ric}}
\newcommand{\tr}{\operatorname{tr}}
\newcommand{\Vol}{\operatorname{Vol}}
\newcommand{\dd}{\mathrm d}
\setlist{itemsep=2pt,topsep=4pt}
\allowdisplaybreaks[2]
\title[Rigidity and stability for $\sigma_2$-curvature]
{Rigidity, sharp inequalities, and stability for $\sigma_2$-curvature by a
reference curvature technique}
\author{Wangzhe Wu}
\address[Wangzhe Wu]{School of Mathematics and Statistics,
Ningbo University, Ningbo, Zhejiang, People's Republic of China}
\email{wuwz18@mail.ustc.edu.cn}
\date{}
\begin{document}
\begin{abstract}
Using a reference-curvature technique, we prove four main theorems.
\begin{enumerate}[label=\textup{(\roman*)},leftmargin=*,
                  itemsep=2pt,topsep=3pt]
\item For $n\geq4$, we remove the mean-curvature pinching from the
Case--Wang rigidity theorem for $\sigma_2$-flat conformal hemispheres
with $A_g\in\overline{\Gamma_2^+}$ and positive constant natural
boundary curvature.

\item On every closed positive Einstein manifold of dimension
$n\geq5$, each conformal metric of constant $\sigma_2$-curvature is
Einstein, and is homothetic to the background unless the background
is a round sphere, thereby resolving the $\sigma_2$-Obata problem
posed by Chang, Gursky, and Yang.

\item We extend the spherical constant-$\sigma_2/\sigma_1$
classification and sharp integral comparison to closed positive
Einstein manifolds. The comparison holds throughout the
positive-scalar-curvature conformal class, with equality only for
homotheties on nonround backgrounds.

\item We establish global quantitative $\sigma_2$ stability on every
fixed nonround positive Einstein background of dimension $n\geq5$,
throughout its positive-scalar-curvature conformal class.
\end{enumerate}
\end{abstract}
\maketitle

\section{Introduction}

Obata's divergence identity~\cite{Obata71} provides a model for
second-order rigidity arguments in conformal geometry. For a metric
conformal to a closed Einstein manifold, it converts the
constant-scalar-curvature equation into the vanishing of a weighted
integral of the squared trace-free Ricci tensor. Hence the conformal
metric is Einstein. On a connected manifold of dimension $n\ge3$,
it is unique up to scaling unless the background is a round sphere,
where conformal transformations give the additional
solutions~\cite{Obata62,Obata71}; see also
Lee and Parker~\cite{LeeParker87}.

Differential identities also have applications to prescribed curvature,
Liouville theorems, and CR geometry. Kazdan--Warner-type integral identities give necessary conditions
for prescribed curvature; see Kazdan--Warner~\cite{KW74} for
Gaussian curvature on the two-sphere and
Bourguignon--Ezin~\cite{BourguignonEzin87} for scalar curvature in
a conformal class. Other applications
include Liouville theorems for semilinear
equations~\cite{BVV91, GidasSpruck81}. In CR geometry, notable examples
are the identities of Jerison and Lee~\cite{JerisonLee88}; see also
X.-N.~Ma and Ou~\cite{MaOu23}. Divergence structures also underlie
the weak continuity results of X.-N.~Ma and W.~Wu~\cite{MaWu25}
for the $\sigma_k$-Yamabe operator in the locally conformally flat setting.

Gursky~\cite{Gursky97} developed a fourth-order counterpart of
Obata's identity and used it to classify critical metrics for
determinant functionals in the conformal class of the round
four-sphere. On positive Einstein backgrounds,
V\'etois~\cite{Vetois23} proved rigidity for metrics with constant
Branson $Q$-curvature. Case~\cite{Case24} extended the argument to
linear combinations of $Q$- and $\sigma_2$-curvatures, with
applications to sharp Sobolev inequalities, while
M.~Li and J.~Wei~\cite{LiWei25} obtained further Liouville-type results
from the identity. Related Liouville and classification results
for higher-order semilinear equations can be found
in~\cite{Lin98,MWW25,Ma26,WeiXu99}.
In the companion paper~\cite{CWSharp}, we established the sharp
lower threshold for mixed $Q$-curvature rigidity in positive
Einstein conformal classes.

The reference curvature technique introduced in~\cite{CWSharp}
provides a common approach to the results below. It inserts a
reference scalar curvature into the Newton and Obata divergence
formulas~\cite[Propositions~2.1 and~2.2]{CWSharp}. The key step here
is to choose a constant reference adapted to each curvature equation,
so that the resulting integral identity has nonnegative terms.
Applications include boundary estimates and rigidity on the hemisphere.
On closed positive Einstein backgrounds, the same technique yields
the curvature and quotient rigidity results.
The latter results lead to sharp integral comparisons and global
quantitative stability. Section~\ref{sec:identities} records the
reference-curvature identity, including its function-valued form
for the prescribed-data extensions.

\paragraph{\textbf{Curvature notation.}}
Throughout the paper, $n$ denotes the dimension of the manifold.
A closed manifold is compact without boundary. For $n\ge3$, set
\begin{equation}\label{id:notation}
 A_g\coloneqq\frac1{n-2}\left(\Ric_g-\frac{R_g}{2(n-1)}g\right),
 \qquad E_g\coloneqq\Ric_g-\frac{R_g}{n}g=(n-2)A_g^\circ.
\end{equation}
For a symmetric covariant tensor $T$, $g^{-1}T$ denotes the
endomorphism obtained by raising one index,
$\tr_gT$ its trace, and
$T^\circ\coloneqq T-(\tr_gT)g/n$ its trace-free part.
If $\lambda_1,\ldots,\lambda_n$ are the eigenvalues of $g^{-1}A_g$,
then
\[
 \sigma_k(A_g)\coloneqq
 \sum_{1\le i_1<\cdots<i_k\le n}\lambda_{i_1}\cdots\lambda_{i_k}.
\]
In particular, $\sigma_1(A_g)=R_g/(2(n-1))$;
the $\sigma_2$ decomposition is given in \eqref{id:sigma}.
Tensor powers and products use the associated endomorphisms, and
all traces, contractions, and norms use the indicated metric.
We write
\[
 \Gamma_k^+\coloneqq\{\lambda\in\mathbb R^n:
              \sigma_j(\lambda)>0\text{ for }1\le j\le k\}.
\]
Thus $A_g\in\Gamma_k^+$ means that its eigenvalue vector lies in
this cone; $\overline{\Gamma_2^+}$ denotes its closure.

\paragraph{\textbf{Differential operators and boundary geometry.}}
We denote the connection and gradient by $\nabla_g$, the Hessian
by $\nabla_g^2$, and set
$\Delta_g\coloneqq\tr_g\nabla_g^2=\operatorname{div}_g\nabla_g$;
thus $-\Delta_g$ has nonnegative spectrum on a closed manifold.
For a covariant two-tensor,
$(\operatorname{div}_gT)_j\coloneqq(\nabla_g)^iT_{ij}$.
Repeated upper and lower indices are summed, and the metric
identifies one-forms with vectors when required.
The measures $\dd v_g$ and $\dd\sigma_g$ are the volume and induced
boundary area measures, respectively. We write
$\Vol_g(M)\coloneqq\int_M\dd v_g$ and
$\operatorname{Area}_g(\Sigma)\coloneqq\int_\Sigma\dd\sigma_g$.
We denote by $H_g$ the mean curvature of the boundary
$\Sigma=\partial M$ as a hypersurface in $(M,g)$.
It is the sum of the boundary's principal curvatures with respect
to the outward unit normal, divided by $n-1$.
With this convention, $H_\delta=1$ on $\partial\mathbb B^n$
for the Euclidean metric $\delta$.
We write $\bar g=g|_{T\Sigma}$ for the metric induced on the boundary
$\Sigma$. We retain the metric subscript in ambient differential operators,
as in~\cite{CWSharp}. On the boundary, we use the abbreviations
$\bar\nabla\coloneqq\nabla_{\bar g}$ and
$\overline{\operatorname{div}}\coloneqq\operatorname{div}_{\bar g}$.

\paragraph{\textbf{Conformal metrics and function spaces.}}
Our conformal convention is
\[
 [g_0]\coloneqq\{u^2g_0:u\in C^\infty(M),\ u>0\}.
\]
The function $u$ is the conformal length factor. This convention
agrees with~\cite{LiWei25,CWSharp}; we write $R_{g_0}$ for the
background scalar curvature. In the closed-manifold results,
$(M,g_0)$ is connected and $g_0$ is a positive Einstein metric,
meaning $\Ric_{g_0}=(R_{g_0}/n)g_0$ with $R_{g_0}>0$.
The unit Euclidean ball is $\mathbb B^n$, with metric $\delta$;
$\mathbb S^n_+$ is the closed unit round hemisphere, with metric
$g_{\mathrm{rd}}$. Smoothness on the ball or hemisphere is
understood up to the boundary. We use the usual spaces $C^j$,
$C^{j,\mu}$, $0<\mu<1$, and $C^\infty$, and set
\[
 \|f\|_{L^p(g)}^p\coloneqq\int_M|f|^p\,\dd v_g,
 \qquad
 \|f\|_{W^{1,p}(g)}^p\coloneqq
        \int_M(|f|^p+|\nabla_gf|_g^p)\,\dd v_g
 \quad(1\le p<\infty),
\]
with $H^1\coloneqq W^{1,2}$ and the usual essential-supremum norm
for $L^\infty$. A norm subscript $p$ abbreviates $L^p$ in the
metric specified locally. The letters $c,C$, with subscripts when
needed, denote constants in estimates; their dependence is stated
locally. Prescribed curvature values are denoted by $\Lambda$,
and $F'$ denotes differentiation in the scalar argument of $F$.
The stability factor $z$ satisfies $u=z^{4/(n-4)}$.

\subsection{A mean-curvature estimate and boundary rigidity}

Escobar's scalar-flat boundary rigidity and sharp trace
inequality~\cite{Escobar92} provide the classical model for the
boundary problem. In that setting, the ordinary mean curvature is
also the natural boundary curvature associated with the scalar
curvature. For higher-order Schouten curvature equations, however,
one must distinguish the variationally natural boundary curvature
associated with $\sigma_k$ from the boundary condition prescribing
the ordinary mean curvature.

The natural boundary curvature $\mathcal B_k$ associated with
$\sigma_k$ was introduced by S.-y.~S.~Chen~\cite{SChen09}, and its
variational structure was developed systematically by
Case and Y.~Wang~\cite{CW18}. On an umbilic boundary and for $k=2$,
$\mathcal B_2$ coincides with the curvature $H_2$ used in the
present paper. Within this natural boundary-curvature framework,
W.~Wei~\cite{Wei26} established a Liouville-type classification for
conformal metrics with positive constant $\sigma_k$-curvature and
constant $\mathcal B_k$ on the half-space. Chu, Y.~Y.~Li, and
Z.~Li~\cite{ChuLiLi} subsequently removed the asymptotic condition
at infinity imposed in that classification.

A separate line of work prescribes the ordinary mean curvature
$H_g$, rather than the natural boundary curvature $\mathcal B_k$.
A.~Li and Y.~Y.~Li~\cite{LiLi06Boundary} formulated the corresponding
fully nonlinear Yamabe problem. They also proved Liouville theorems
on the Euclidean half-space. Chu, Y.~Y.~Li, and Z.~Li~\cite{ChuLiLiMean24} subsequently
obtained a half-space classification under optimal structural
conditions on the elliptic cone. On the ball, X.~Chen and W.~Wei
~\cite{ChenWeiBall} classified conformal metrics with constant
$\sigma_k$-curvature and constant ordinary mean curvature. Their
result includes the zero-curvature branch
$\sigma_k(A_g)=0$, in which the metric is forced to be the standard
flat ball.

For this ordinary mean-curvature boundary problem, the analytic
theory of a priori estimates, compactness, and existence was
developed in several works. Jin, A.~Li, and Y.~Y.~Li~\cite{JinLiLi07}
established fundamental local estimates and existence results.
Y.~Y.~Li and Nguyen~\cite{LiNguyen14Boundary} proved existence and
compactness on locally conformally flat manifolds with umbilic boundary.
More recently, X.~Chen and W.~Wei~\cite{ChenWeiBoundary25} obtained local
boundary $C^2$ estimates for the $\sigma_2$ equation with the
nonlinear Neumann condition corresponding to prescribed ordinary
mean curvature, together with existence results in dimensions three
and four under suitable geometric hypotheses. For the modified
Schouten tensor equation, they also proved boundary and global
$C^2$ estimates~\cite{ChenWeiModified24}.

A boundary is called \emph{umbilic} if all its principal curvatures
are equal at each point, equivalently, if its second fundamental
form satisfies $\mathrm{II}=H_g\bar g$. The value of $H_g$ may vary
along the boundary. On an umbilic boundary, the natural boundary term is
\[
 H_2(g)\coloneqq
 H_g\tr_{\bar g}(A_g|_{T\Sigma})+\frac{n-1}{3}H_g^3,
 \qquad \bar g\coloneqq g|_{T\Sigma}.
\]
To relate the boundary curvature terms to the interior $\sigma_2$-curvature, we use the conformal Killing field
\[
X\coloneqq-\nabla_{g_{\mathrm{rd}}}x_{n+1}
\]
on $\mathbb S^n_+\subset\mathbb R^{n+1}$, as in
Case and Y.~Wang~\cite[proof of Lemma~5.4]{CW18}. For
$g=u^2g_{\mathrm{rd}}$, its conformal factor is
\[
\chi\coloneqq\frac1n\operatorname{div}_gX
=x_{n+1}+X(\log u).
\]

\begin{maintheorem}[A mean-curvature estimate and boundary rigidity]
\label{main:boundary}
\label{bnd:unpinched}
\label{bnd:quantitative}
Let $g=u^2g_{\mathrm{rd}}$ be a smooth positive conformal metric on
$M\coloneqq \mathbb S^n_+$, $n\geq4$, with $A_g\in\overline{\Gamma_2^+}$.
Write $\Sigma\coloneqq \partial M$. Suppose that $H_2(g)=\Lambda>0$
is constant on $\Sigma$, and set
\[
 h\coloneqq \left(\frac{3\Lambda}{n-1}\right)^{1/3}.
\]
Let $\chi$ be as in \eqref{bnd:CK-field}. Then the following conclusions hold.
\begin{enumerate}[label=\textup{(\roman*)},leftmargin=*]
\item The mean-curvature deviation satisfies
\begin{align}
 \label{bnd:defect-estimate}
 \Lambda\int_\Sigma\left(1-\frac{H_g}{h}\right)^2\,\mathrm{d}\sigma_g
 &\leq\frac{4}{3\sqrt3}
       \int_M\chi\sigma_2(A_g)\,\dd v_g\notag\\
 &\leq\frac{4}{3\sqrt3}\|\chi\|_{L^\infty(M)}
       \int_M\sigma_2(A_g)\,\dd v_g.
\end{align}
\item If $\sigma_2(A_g)=0$ in $M$, then $(M,g)$ is isometric to a
Euclidean ball.
\end{enumerate}
\end{maintheorem}

\paragraph{\textbf{Comparison with related work.}}
Case and Y.~Wang~\cite[Theorem~1.2]{CW18} proved flatness for
$\sigma_2(A_g)=0$ and positive constant $H_2$ under the condition
$\sup_\Sigma H_g\le3\inf_\Sigma H_g$.
Theorem~\ref{main:boundary}\textup{(ii)} removes this condition.
Part~\textup{(i)} also controls the mean-curvature deviation when
the interior $\sigma_2$ curvature does not vanish.

X.~Chen and W.~Wei~\cite[Theorem~1]{ChenWeiObata26} recently proved an
unpinched rigidity theorem for $\sigma_k$-flat conformal metrics on
the round hemisphere with positive constant associated boundary
curvature, for $k\ge2$. For $k=2$, their $\mathcal B_2$ agrees with
our $H_2$, and their result overlaps with
Theorem~\ref{main:boundary}\textup{(ii)}. These overlapping results
were obtained independently, by different methods. Their
Obata--Escobar argument uses an auxiliary function arising from
a conformal compactification of the hyperbolic metric. Our proof
uses the reference-curvature Newton identity on the boundary and
gives the explicit mean-curvature estimate~\eqref{bnd:defect-estimate},
retaining the interior $\sigma_2$-curvature term.

\subsection{Constant \texorpdfstring{$\sigma_2$}{sigma-two}
curvature on closed manifolds}

For the $\sigma_k$-Yamabe problem on closed locally conformally flat
manifolds, Guan and G.~Wang introduced a fully nonlinear
conformal flow~\cite{GuanWangFlow03}. A.~Li and Y.~Y.~Li developed
an existence and compactness theory for conformally invariant fully
nonlinear equations~\cite{LiLi03,LiLi05}. These works concern the
analytic existence theory. Our focus here is the rigidity of
constant-$\sigma_2$ metrics in a fixed positive Einstein conformal
class.

For the rigidity problem, Viaclovsky~\cite{V00} classified
admissible constant-$\sigma_k$ metrics in the conformal class of
the round sphere; lifting to the universal cover gives the
corresponding classification on spherical space forms. In dimension
four, Gursky and Streets~\cite[Theorem~1.5]{GS18} proved uniqueness for
the admissible $\sigma_2$-Yamabe problem, which in particular applies
to positive Einstein conformal classes.

\paragraph{\textbf{The problem of Chang, Gursky, and Yang.}}
Chang, Gursky, and Yang~\cite[Introduction]{CGYEntire03} suggested that
the Obata-type rigidity argument for the $\sigma_2$ equation on the
sphere might extend to general Einstein manifolds. Thus, on a closed
positive Einstein manifold, one asks whether every conformal metric
of constant $\sigma_2$-curvature must itself be Einstein.
Theorem~\ref{main:constant} answers this question affirmatively in
dimensions $n\geq5$.

\begin{maintheorem}[Constant \(\sigma_2\) curvature]
\label{main:constant}\label{clo:constant}\label{weyl:sigma2}
Let $(M^n,g_0)$ be closed and connected, $n\ge5$, with $g_0$
Einstein and $R_{g_0}>0$. Let $g=u^2g_0$, where $u$ is smooth and
positive. If
\[
 \sigma_2(A_g)=\Lambda
\]
for a constant $\Lambda$, then $\Lambda>0$, $R_g>0$, and $g$ is Einstein.
If $(M,g_0)$ is not isometric, up to scaling, to the round sphere,
then $u$ is constant.
\end{maintheorem}

\paragraph{\textbf{Comparison with earlier work.}}
In dimension three, every positive Einstein manifold is a spherical
space form, while the four-dimensional admissible case is covered
by Gursky and Streets~\cite[Theorem~1.5]{GS18}. We therefore state
Theorem~\ref{main:constant} for $n\geq5$. Neither scalar-curvature
positivity nor admissibility of the conformal metric is assumed;
both follow from the constant-$\sigma_2$ equation.

The Einstein assumption in Theorem~\ref{main:constant} cannot be
removed. Indeed, Deng, Lu, and J.~Wei~\cite{DLW26} constructed, for every
$n\geq27$, a fixed non-Einstein metric $g_0$ on $\mathbb S^n$
whose conformal class contains a noncompact family of positive
$\Gamma_2^+$-admissible constant-$\sigma_2$ metrics. Although their
theorem states that $g_0$ is not locally conformally flat, the
non-Einstein property also follows from their construction: the
metric agrees with the round metric on a nonempty open set, while
unique continuation would force any Einstein metric with this
property to be round. Thus the rigidity and uniqueness conclusion
fails for general background metrics.

\subsection{Curvature quotient rigidity and sharp comparison}

The curvature-quotient problem has several closely related aspects.
The conformal rigidity theorem of A.~Li and Y.~Y.~Li~\cite[Corollary~1.6]{LiLi03}
includes the classification of admissible solutions of the constant
$\sigma_2/\sigma_1$ equation on the round sphere. Guan and
G.~Wang~\cite{GuanWang04} subsequently established sharp integral
inequalities for general curvature quotients on locally conformally
flat manifolds. The corresponding constant-quotient equations arise
as Euler--Lagrange equations. For the particular quotient
$\sigma_2/\sigma_1$, Ge and G.~Wang~\cite{GW13} removed the local
conformal flatness assumption from the variational existence theory:
they proved that the conformal quotient admits a minimizer in the
positive-scalar-curvature class, with the minimizer lying in the
admissible cone. On the sphere, this yields the sharp comparison
recorded in~\cite[(2.1)]{FP24}.

More recently, Ge, G.~Wang, and W.~Wei~\cite[Theorem~1.1]{GWWFlow26} proved sharp spherical inequalities for general integral curvature quotients in the larger $\Gamma_{k-1}^+$-admissible class. In the case $(k,l)=(2,1)$ and $n\ge5$, this means that the sharp spherical $\sigma_2/\sigma_1$ comparison holds under the sole assumption of positive scalar curvature. Related recent developments include their existence theory for the $\sigma_2$-Yamabe problem in the positive-scalar-curvature class~\cite{GWW26}, as well as their Yamabe-type variational problem for the quotient of $Q$ curvature by scalar curvature~\cite{GWWQuot26}. These results provide the variational framework surrounding the quotient problem, while our contribution is the rigidity and resulting sharp comparison in general positive Einstein conformal classes.

Write
\begin{equation}\label{intro:normalized-functionals}
 \mathcal F_1(g)\coloneqq
 \frac{\int_M\sigma_1(A_g)\,\dd v_g}{\Vol_g(M)^{(n-2)/n}},
 \qquad
 \mathcal F_2(g)\coloneqq
 \frac{\int_M\sigma_2(A_g)\,\dd v_g}{\Vol_g(M)^{(n-4)/n}}.
\end{equation}

\begin{maintheorem}[Curvature quotient rigidity and sharp comparison]
\label{main:quotient}\label{clo:quotient}
\label{sharp:comparison}\label{main:sharp}
Let $(M^n,g_0)$ be closed and connected, $n\ge4$, with $g_0$
Einstein, $R_{g_0}>0$, and $g_0$ not locally conformally flat.
Let $g=u^2g_0$, where $u$ is smooth and positive.
\begin{enumerate}[label=\textup{(\roman*)}]
\item If $\sigma_2(A_g)=\Lambda\sigma_1(A_g)$ for a constant $\Lambda$, then
$\Lambda>0$, $R_g>0$, and $g=s^2g_0$. More precisely,
\[
 \Lambda=\frac{R_g}{4n},\qquad s^2=\frac{R_{g_0}}{4n\Lambda}.
\]
\item If $n\ge5$, every smooth conformal metric $g$ with $R_g>0$ satisfies
\begin{equation}\label{sharp:strong}
 \frac{\mathcal F_2(g)}{\mathcal F_2(g_0)}
 \ge\left(\frac{\mathcal F_1(g)}{\mathcal F_1(g_0)}\right)^{\frac{n-4}{n-2}}
 \ge1.
\end{equation}
In particular,
\begin{equation}\label{sharp:pure}
 \frac{\displaystyle\int_M\sigma_2(A_g)\,\dd v_g}
      {\Vol_g(M)^{(n-4)/n}}
 \ge\frac{R_{g_0}^2}{8n(n-1)}\Vol_{g_0}(M)^{4/n}.
\end{equation}
Equality in either comparison in \eqref{sharp:strong}, or in
\eqref{sharp:pure}, holds exactly when $g=s^2g_0$ for some $s>0$.
\end{enumerate}
\end{maintheorem}

\paragraph{\textbf{Comparison with earlier work.}}
Together with the spherical classification, part~\textup{(i)}
extends the quotient rigidity theorem of A.~Li and Y.~Y.~Li to arbitrary
closed positive Einstein manifolds, including those with nonvanishing
Weyl curvature. Part~\textup{(ii)} extends the sharp comparison of
Guan and G.~Wang by replacing the local conformal flatness assumption
with the Einstein condition.
More precisely, if $n\ge5$ and $(M^n,g_0)$ is closed and positive
Einstein, then the comparison holds for every smooth metric
$g\in[g_0]$ with positive scalar curvature, without assuming
$\sigma_2(A_g)>0$. The result is global in the conformal class and
requires no closeness to $g_0$, in contrast with the small $C^2$
perturbative results of J.~Chen, Fang, and Zhong~\cite{CFZ26} under
pointwise curvature-quotient bounds near strictly stable positive
Einstein metrics.

\subsection{Global quantitative stability}\label{intro:stability}

We next turn from rigidity to quantitative stability: whether a small deficit in the sharp comparison forces the conformal metric to be close to the extremal one. This principle goes back to the quantitative Sobolev inequality of
Bianchi and Egnell~\cite{BE91}. Engelstein, Neumayer, and
Spolaor~\cite{ENS22} developed it in the Yamabe setting. Related local stability results for curvature functionals near Einstein metrics were obtained by Gursky and Viaclovsky~\cite{GV15}. For the $\sigma_2$ problem, Frank and Peteranderl~\cite{FP24} proved a global stability estimate on the round sphere under the sole assumption of positive scalar curvature. More recently, Peteranderl~\cite[Theorem~1]{Peteranderl26} established quantitative $\sigma_k$ stability on locally conformally flat conformal classes within the $k$-admissible cone. Our result extends the global $\sigma_2$ stability theory from the round sphere to nonround positive Einstein backgrounds.

We use the fourth-order conformal factor $z$, related to the
length factor by $u=z^{4/(n-4)}$.

\begin{maintheorem}[Global quantitative stability]
\label{stab:global}
\label{main:stability}
Let $(M^n,g_0)$ be closed and connected, $n\ge5$, with $g_0$
Einstein, $R_{g_0}>0$, and $\Vol_{g_0}(M)=1$. Assume that $(M,g_0)$
is not isometric, up to scaling, to the round sphere.
There exists $c_*=c_*(M,g_0)>0$ such that every smooth positive $z$
with
\[
 g=z^{8/(n-4)}g_0,\qquad
 \int_Mz^{4n/(n-4)}\,\dd v_{g_0}=1,\qquad R_g>0,
\]
obeys
\begin{equation}
 \mathcal F_2(g)-\mathcal F_2(g_0)
 \ge c_*\left(\|z-1\|_{H^1(g_0)}^2+\|z-1\|_{W^{1,4}(g_0)}^4\right).
 \label{stab:estimate}
\end{equation}
We do not assume $\sigma_2(A_g)>0$ or a uniform positive lower bound
for $z$. The constant depends only on the fixed background metric.
\end{maintheorem}

\subsection{Method of proof}

Choosing the reference scalar curvature to match the algebraic
structure of the equation is the common step in Theorems~A--C.
In each of these theorems, $r$ is constant. Its value is chosen
so that the reference-curvature identity~\eqref{id:variable-formula}
yields a nonnegative integral or a boundary defect estimate.

For Theorem~A, we take
$r=(n-1)(n-2)h^2$, where
$h=(3\Lambda/(n-1))^{1/3}$. We combine~\eqref{id:variable-formula}
with the $T_2$ flux formula and the Gauss equation. This yields the
boundary defect estimate and determines the boundary geometry in
the equality case. Once the boundary is shown to be round, the
rigidity theorem of Shi and Tam~\cite[Theorem~4.1]{ShiTam02}
completes the proof of the equality case.

For Theorem~B, choosing
$r=\sqrt{8n(n-1)\Lambda}$
in~\eqref{id:variable-formula} produces a nonnegative integral identity for the constant-$\sigma_2$ equation and forces the conformal metric to be Einstein.

For Theorem~C, the choice
$r=4n\Lambda$
in~\eqref{id:variable-formula} gives the rigidity of the constant $\sigma_2/\sigma_1$ equation. Together with the minimizer theorem of Ge and G.~Wang~\cite{GW13},
this rigidity result identifies the extremal metric, yielding the
sharp global comparison.

Theorem~D is built on Theorem~C. The sharp comparison first controls
the Yamabe deficit; the stability and compactness theory of
Engelstein, Neumayer, and Spolaor~\cite{ENS22} then gives quadratic
$H^1$ control, while Case's conformal energy identity~\cite{Case21}
controls the quartic gradient term.

Allowing the reference to depend on the conformal factor gives
the nonincreasing prescribed-curvature extensions in
Remarks~\ref{bnd:monotone-extension} and~\ref{clo:monotone-extension}.

\section{Reference curvature identities}\label{sec:identities}

Let $(M^n,g_0)$ be an Einstein manifold, $n\geq3$, with
$\Ric_{g_0}= \frac{R_{g_0}}{n}g_0$, and let $g=u^2g_0$, where $u>0$ is smooth.
No compactness or sign assumption is needed in the identities of this section.
We use the notation and curvature conventions of~\cite[Section~2]{CWSharp}.
Differential operators carry their metric subscripts; unless otherwise
stated, contractions and norms refer to $g$. We retain the notation in \eqref{id:notation}.

\paragraph{\textbf{Bianchi and integration formulas.}}
With the divergence convention fixed above, the contracted second
Bianchi identity and its equivalent forms are
\begin{align}
 (\nabla_g)^i(\Ric_g)_{ij}&=\frac12(\nabla_g)_jR_g,\label{id:bianchi-ric}\\
 (\nabla_g)^i(A_g)_{ij}&=(\nabla_g)_j\sigma_1(A_g)
       =\frac1{2(n-1)}(\nabla_g)_jR_g,\label{id:bianchi-schouten}\\
 (\nabla_g)^i(E_g)_{ij}&=\frac{n-2}{2n}(\nabla_g)_jR_g.\label{id:bianchi-tracefree}
\end{align}
See Besse~\cite[Chapter~1]{Besse87}; the formulas also appear in
M.~Li and J.~Wei~\cite[Section~2]{LiWei25}.
These formulas hold for every smooth metric in dimension $n\ge3$.
In particular, $E_g=0$ implies $\dd R_g=0$ on each connected component.

For a symmetric two-tensor $T$ and a smooth function $f$, the product
rule gives~\cite[proof of Theorem~2.1]{LiWei25}
\begin{equation}\label{id:tensor-product}
 \operatorname{div}_g(T\nabla_g f)
   =\langle\operatorname{div}_gT,\nabla_g f\rangle_g+\langle T,\nabla_g^2 f\rangle_g.
\end{equation}
We integrate \eqref{id:tensor-product} using the divergence
theorem with the outward unit normal and the measures of $g$.
On a closed manifold the boundary flux vanishes.

\paragraph{\textbf{General conformal formulas.}}
For a smooth function $f$ and $\widehat g=e^{2f}g$, the basic
transformation laws, with all derivatives on the right taken in $g$,
are given in~\cite[Chapter~1]{Besse87}:
\begin{align}
 A_{\widehat g}&=A_g-\nabla_g^2 f+\dd f\otimes\dd f
                   -\frac12|\nabla_g f|_g^2g,
                   \label{id:conformal-schouten}\\
 \sigma_1(A_{\widehat g})&=e^{-2f}\left(
       \sigma_1(A_g)-\Delta_g f-\frac{n-2}{2}|\nabla_g f|_g^2\right).
                   \label{id:conformal-trace}
\end{align}
The Schouten decomposition is recorded in
M.~Li and J.~Wei~\cite[Section~2]{LiWei25}:
\begin{equation}\label{id:sigma}
 \sigma_2(A_g)=\frac{R_g^2}{8n(n-1)}-\frac{|E_g|^2}{2(n-2)^2}.
\end{equation}
The first Newton tensor is
\begin{equation}\label{id:newton}
 T_1(A_g)\coloneqq \sigma_1(A_g)g-A_g=\frac{R_g}{2n}g-\frac{E_g}{n-2},
 \qquad \operatorname{div}_gT_1(A_g)=0.
\end{equation}
Its divergence vanishes for every smooth metric by
\eqref{id:bianchi-schouten}. Related conformal divergence formulas
involving Newton tensors are given by
Ge, G.~Wang, and W.~Wei~\cite[Proposition~5.3]{GWWFlow26}; their $k=2$
formula also holds without local conformal flatness.
We regard a symmetric tensor as an endomorphism when applying
it to a gradient.

For an Einstein background, we use the following conformal
identities from M.~Li and J.~Wei~\cite[Section~2]{LiWei25}:
\begin{align}
 (\nabla_g^2 u)^\circ&=-\frac{1}{n-2}uE_g,\label{id:hess}\\
 \Delta_g u&=\frac{n}{2}u^{-1}|\nabla_g u|^2-\frac{1}{2(n-1)}uR_g
                  +\frac{1}{2(n-1)}u^{-1}R_{g_0}.\label{id:trace}
\end{align}
As in~\cite[Section~2]{CWSharp}, apply
\eqref{id:conformal-schouten} to $g_0=u^{-2}g$ to obtain
\[
 \frac{R_{g_0}}{2n(n-1)}u^{-2}g
 =A_g+u^{-1}\nabla_g^2u-\frac12u^{-2}|\nabla_g u|^2g.
\]
Taking the trace-free part and the trace with respect to $g$ gives
\eqref{id:hess} and \eqref{id:trace}, respectively. In particular,
the derivatives in these identities use the target metric $g$.

\paragraph{\textbf{The reference identity used below.}}
For a constant reference $r$, the following Newton identity is
proved in~\cite[Proposition~2.1]{CWSharp}. The same computation,
based on \eqref{id:newton}, \eqref{id:hess}, \eqref{id:trace},
and \eqref{id:sigma}, is valid in every dimension $n\geq3$.
Allowing $r\in C^1(M)$ adds the final term below by the product
rule. Thus, for every real-valued $r\in C^1(M)$,
\begin{equation}\label{id:variable-formula}
\begin{split}
 \operatorname{div}_g\!\left[
 (n-1)\left(T_1(A_g)-\frac{r}{2n}g\right)\nabla_g u
 \right]
 ={}&
 \frac{R_g-r}{4n}
 \left(
 n(n-1)u^{-1}|\nabla_g u|^2
 +ur+u^{-1}R_{g_0}
 \right)
 \\
 &+u\left(
 \frac{r^2}{4n}
 -2(n-1)\sigma_2(A_g)
 \right)
 -\frac{n-1}{2n}
 \langle\nabla_g r,\nabla_g u\rangle_g .
\end{split}
\end{equation}
This pointwise identity requires no curvature equation, compactness,
or sign condition. The gradient term vanishes when $r$ is constant.
In each application below, we choose the reference $r$ before
integrating~\eqref{id:variable-formula}.

\section{Mean-curvature estimates and boundary rigidity}
\label{bnd:section}

The proof of Theorem~\ref{main:boundary} has three steps. We first
derive a boundary identity whose nonnegative remainder gives the
mean-curvature estimate. The equality case determines the boundary
geometry. The rigidity statement in the boundary mean-curvature
inequality of Shi and Tam then shows that the filling is
a Euclidean ball.

Throughout this section, $n\geq4$ and
\[
 M\coloneqq\mathbb S^n_+,
 \qquad
 \Sigma\coloneqq\partial M=\mathbb S^{n-1},
 \qquad
 g=u^2g_{\mathrm{rd}},
 \qquad
 u\in C^\infty(M),\quad u>0,
\]
where $g_{\mathrm{rd}}$ has sectional curvature one and all data are
smooth up to the boundary. We use the boundary formulas of
Case and Y.~Wang~\cite{CW18}, with their boundary dimension replaced by
$n-1$.

Let $\nu_g\in TM|_\Sigma$ be the outward $g$-unit normal to
$\Sigma=\partial M$ within the hemisphere $M=\mathbb S^n_+$, and
let $H_g$ denote the normalized mean curvature:
\[
 \mathrm{II}_g(Y,Z)=g((\nabla_g)_Y\nu_g,Z),
 \qquad
 H_g=\frac1{n-1}\tr_{\bar g}\mathrm{II}_g,
 \qquad
 \bar g\coloneqq g|_{T\Sigma},
\]
for $Y,Z\in T\Sigma$.
With this convention, the boundary of the Euclidean unit ball has
mean curvature one. We also write
$\nu_gf=\partial_{\nu_g}f=\dd f(\nu_g)$.

The equator is totally geodesic, and hence umbilic, with respect to
$g_{\mathrm{rd}}$. Since the trace-free second fundamental form is
conformally covariant, it remains umbilic for
$g=u^2g_{\mathrm{rd}}$. Consequently,
\[
 \mathrm{II}_g=H_g\bar g.
\]
No constancy of $H_g$ is assumed.

\paragraph{\textbf{Boundary formulas.}}
Set $T_2(A_g)\coloneqq\sigma_2(A_g)g-\sigma_1(A_g)A_g+A_g^2$.
We use the standard Newton identities~\cite[Section~2.1]{CW18}
\begin{equation}\label{bnd:newton-traces}
\begin{aligned}
 \tr_gT_1(A_g)&=(n-1)\sigma_1(A_g),
 &\tr_g(A_gT_1(A_g))&=2\sigma_2(A_g),\\
 \tr_gT_2(A_g)&=(n-2)\sigma_2(A_g).
\end{aligned}
\end{equation}
Products are understood as endomorphism products after raising an index.
We write $A_g\in\overline{\Gamma_2^+}$ when its eigenvalues lie in the closure
of $\{\lambda:\sigma_1(\lambda)>0,\ \sigma_2(\lambda)>0\}$.
On this cone, $T_1(A_g)\geq0$ by~\cite[Section~2.1]{CW18} and continuity.

Set
\begin{equation}
 \label{bnd:intrinsic-notation}
 B\coloneqq A_g|_{T\Sigma},\qquad
 \sigma_1(B):= \tr_{\bar g}B=T_1(A_g)(\nu_g,\nu_g),\qquad
 B^\circ\coloneqq B-\frac{\sigma_1(B)}{n-1}\bar g.
\end{equation}
Thus $B^\circ$ is the trace-free part of the tangential Schouten tensor.
The equality $\sigma_1(B)=T_1(A_g)(\nu_g,\nu_g)$ is the $k=1$ case of
\cite[proof of Lemma~5.1]{CW18}. Bars on differential operators and norms
refer to $\bar g$.

The Gauss--Codazzi equations~\cite[Chapter~1]{Besse87} simplify under
local conformal flatness and umbilicity to
\begin{align}
 \operatorname{Sec}_{\bar g}(Y\wedge Z)
   &=A_g(Y,Y)+A_g(Z,Z)+H_g^2,\label{bnd:gauss-sectional}\\
 \Ric_g(\nu_g,Y)&=-(n-2)Y(H_g),&
 A_g(\nu_g,Y)&=-Y(H_g).\label{bnd:codazzi-normal}
\end{align}
Here $Y,Z$ are orthonormal tangent vectors and $\operatorname{Sec}$ is
sectional curvature. Taking the trace over tangent directions in
\eqref{bnd:gauss-sectional} gives the intrinsic Ricci tensor; taking its
trace gives the intrinsic scalar curvature:
\[
 \Ric_{\bar g}=(n-3)B+\bigl(\sigma_1(B)+(n-2)H_g^2\bigr)\bar g,
 \qquad R_{\bar g}=2(n-2)\sigma_1(B)+(n-1)(n-2)H_g^2.
\]
Equivalently, the Schouten form of the Gauss equation in
Case and Y.~Wang~\cite[equation~(2.4)]{CW18} reads
\begin{equation}
 \label{bnd:gauss}
 A_{\bar g}=B+\frac{H_g^2}{2}\bar g,
 \qquad
 \sigma_1(A_{\bar g})=\sigma_1(B)+\frac{n-1}{2}H_g^2,
 \qquad
 A_{\bar g}-\frac{\sigma_1(A_{\bar g})}{n-1}\bar g=B^\circ.
\end{equation}
The intrinsic Schouten tensor $A_{\bar g}$ is defined by
\eqref{id:notation} in dimension $n-1\ge3$, and
$\sigma_1(A_{\bar g})=\tr_{\bar g}A_{\bar g}$.
The intrinsic contracted Bianchi identity~\eqref{id:bianchi-schouten}
and~\eqref{id:newton} give
\begin{equation}\label{bnd:bianchi}
 \bar\nabla^a(A_{\bar g})_{ab}=\bar\nabla_b\sigma_1(A_{\bar g}),
 \qquad \overline{\operatorname{div}}T_1(A_{\bar g})=0.
\end{equation}
Finally, the umbilic formula~\cite[Lemma~5.2]{CW18} gives
\begin{equation}
 \label{bnd:H2}
 H_2(g)\coloneqq H_g \sigma_1(B)+\frac{n-1}{3}H_g^3.
\end{equation}

\paragraph{\textbf{The reference mean curvature.}}
Suppose that $A_g\in\overline{\Gamma_2^+}$ and $H_2(g)>0$ on
$\Sigma$. Define
\begin{equation}
 \label{bnd:anchor}
 h\coloneqq
 \left(\frac{3H_2(g)}{n-1}\right)^{1/3}.
\end{equation}
For a Euclidean ball, the Schouten tensor vanishes and hence
\[
 H_2=\frac{n-1}{3}H^3.
\]
Thus $h$ is the pointwise Euclidean model mean curvature determined
by $H_2(g)$.

Since $T_1(A_g)\geq0$,
\[
 \sigma_1(B)=T_1(A_g)(\nu_g,\nu_g)\geq0.
\]
It follows from \eqref{bnd:H2} and $H_2(g)>0$ that
\begin{equation}
 \label{bnd:H-upper}
 0<H_g\leq h.
\end{equation}

To convert the boundary term into interior curvature, let
\begin{equation}
 \label{bnd:CK-field}
 X\coloneqq-\nabla_{g_{\mathrm{rd}}}x_{n+1},
 \qquad
 \chi\coloneqq\frac1n\operatorname{div}_gX
 =x_{n+1}+X(\log u).
\end{equation}

\begin{proposition}[The boundary defect identity]
\label{bnd:identity}
With $h$ as in \eqref{bnd:anchor}, set
\begin{align}
 \mathcal R_h
 \coloneqq{}&
 \frac1{n-1}u
 \left[
 \sigma_1(B)^2
 -\sigma_1(A_{\bar g})
 \left(
 \sigma_1(A_{\bar g})-\frac{n-1}{2}h^2
 \right)
 \right]
 \notag\\
 &+\frac12u^{-1}(|\bar\nabla u|^2+1)
 \left(
 \sigma_1(A_{\bar g})-\frac{n-1}{2}h^2
 \right).
 \label{bnd:defect-density}
\end{align}
Then
\begin{align}
 \mathcal R_h
 ={}&
 \frac{n-1}{12H_g}(H_g-h)^2
 \bigl[
 u(H_g^3+2hH_g^2+4h^2H_g+2h^3)
 \notag\\
 &\hspace{37mm}
 +u^{-1}(|\bar\nabla u|^2+1)(H_g+2h)
 \bigr]
 \geq0,
 \label{bnd:remainder}
\end{align}
and
\begin{equation}
 \label{bnd:integrated-identity}
 \int_\Sigma\mathcal R_h\,\mathrm d\sigma_g
 =
 2\int_M\chi\sigma_2(A_g)\,\dd v_g
 +
 \int_\Sigma
 h\langle\bar\nabla h,\bar\nabla u\rangle
 \,\mathrm d\sigma_g.
\end{equation}
Moreover, $\mathcal R_h=0$ at a point if and only if $H_g=h$ there.
\end{proposition}

\begin{proof}
Solving \eqref{bnd:H2} and \eqref{bnd:gauss} for the two traces gives
\begin{equation}
 \label{bnd:anchor-algebra}
 \sigma_1(B)
 =
 \frac{n-1}{3H_g}(h^3-H_g^3),
 \qquad
 \sigma_1(A_{\bar g})
 =
 \frac{n-1}{6H_g}(2h^3+H_g^3).
\end{equation}
Consequently,
\begin{align}
 \sigma_1(A_{\bar g})-\frac{n-1}{2}h^2
 &=
 \frac{n-1}{6H_g}(H_g-h)^2(H_g+2h),
 \notag\\
 \sigma_1(B)^2
 -\sigma_1(A_{\bar g})
 \left(
 \sigma_1(A_{\bar g})-\frac{n-1}{2}h^2
 \right)
 &=
 \frac{(n-1)^2}{12H_g}(H_g-h)^2
 \notag\\
 &\quad\times
 (H_g^3+2hH_g^2+4h^2H_g+2h^3).
 \label{bnd:factorizations}
\end{align}
Substitution into \eqref{bnd:defect-density} proves
\eqref{bnd:remainder}.

We next apply \eqref{id:variable-formula} in dimension $n-1$ to
$\bar g=u^2g_{\mathbb S^{n-1}}$, with
\[
 R_{g_{\mathbb S^{n-1}}}=(n-1)(n-2),
 \qquad
 r=(n-1)(n-2)h^2.
\]
Using
\[
 R_{\bar g}-r
 =
 2(n-2)
 \left(
 \sigma_1(A_{\bar g})-\frac{n-1}{2}h^2
 \right),
 \qquad
 \bar\nabla r
 =
 2(n-1)(n-2)h\bar\nabla h,
\]
together with
\[
 2\sigma_2(B)
 =
 \frac{n-2}{n-1}\sigma_1(B)^2-|B^\circ|^2,
 \qquad
 2\sigma_2(A_{\bar g})
 =
 \frac{n-2}{n-1}\sigma_1(A_{\bar g})^2-|B^\circ|^2,
\]
we obtain
\[
\begin{split}
 2u\sigma_2(B)
 &+
 \overline{\operatorname{div}}
 \left[
 \left(
 T_1(A_{\bar g})-\frac{n-2}{2}h^2\bar g
 \right)(\bar\nabla u)
 \right]
 \\
 &=(n-2)
 \left(
 \mathcal R_h
 -h\langle\bar\nabla h,\bar\nabla u\rangle
 \right).
\end{split}
\]

It remains to integrate this identity. Since $g$ is locally
conformally flat, $\operatorname{div}_gT_2(A_g)=0$. Moreover,
$T_2(A_g)(\nu_g,\nu_g)=\sigma_2(B)$ and
\[
 \frac12\mathcal L_Xg=\chi g,
 \qquad
 X|_\Sigma=u\nu_g.
\]
These identities and the divergence theorem give
\begin{equation}
 \label{bnd:CK-flux}
 \int_\Sigma u\sigma_2(B)\,\mathrm d\sigma_g
 =
 (n-2)\int_M\chi\sigma_2(A_g)\,\dd v_g.
\end{equation}
The tangential divergence integrates to zero, and
\eqref{bnd:integrated-identity} follows.
\end{proof}

\paragraph{\textbf{Flatness of the interior.}}

Once the boundary is shown to be round, the equality rigidity theorem
of Shi and Tam identifies the filling as a Euclidean ball.

\begin{lemma}
\label{bnd:flat-filling}
Suppose that $R_g\ge0$ on $M$, $H_g>0$ on $\Sigma$, and
$B=0$. Then $g$ is flat and $(M,g)$ is isometric to a
Euclidean ball.
\end{lemma}

\begin{proof}
By~\eqref{bnd:gauss}, $A_{\bar g}=\frac12H_g^2\bar g$ and
$\sigma_1(A_{\bar g})=(n-1)H_g^2/2$. The contracted Bianchi
identity~\eqref{bnd:bianchi} therefore gives
\[
 \frac12\dd(H_g^2)=\frac{n-1}{2}\dd(H_g^2),
 \qquad \Rightarrow (n-2)\dd(H_g^2)=0.
\]
Since $n\ge4$, $H_g>0$, and $\Sigma$ is connected, $H_g$ is a
positive constant. Since $B=A_g|_{T\Sigma}=0$, the sectional Gauss equation
\eqref{bnd:gauss-sectional} reduces, for every
$\bar g$-orthonormal pair $Y,Z\in T\Sigma$, to
\[
 \operatorname{Sec}_{\bar g}(Y\wedge Z)=H_g^2.
\]
Thus $(\Sigma,\bar g)$ has constant sectional curvature $H_g^2$.
Since $\Sigma$ is compact and simply connected, it is the round
sphere of radius $H_g^{-1}$.

The manifold $M$ is diffeomorphic to a ball and hence is spin.
The standard Euclidean isometric embedding of $(\Sigma,\bar g)$ is
the sphere of radius $H_g^{-1}$ in $\mathbb R^n$ and is strictly
convex. Its trace mean curvature is
\[
 \mathcal H_0=(n-1)H_g.
\]
On the other hand, since
$\mathrm{II}_g=H_g\bar g$, the trace mean curvature of
$\Sigma\subset(M,g)$ is
\[
 \mathcal H_g=(n-1)H_g=\mathcal H_0.
\]
Since $R_g\geq0$, the boundary mean-curvature inequality of
Shi and Tam applies. The preceding pointwise equality
shows that equality holds in this inequality. Its rigidity statement
~\cite[Theorem~4.1]{ShiTam02} implies that $(M,g)$ is isometric to
a bounded Euclidean domain. The boundary of this domain is connected
and totally umbilic with constant normalized mean curvature
$H_g>0$; hence it is the round sphere of radius $H_g^{-1}$.
Therefore the domain is the corresponding Euclidean ball.
\end{proof}

For constant positive boundary data, the reference $h$ is constant.
Theorem~\ref{main:boundary} follows from the boundary identity and
its equality case, with Lemma~\ref{bnd:flat-filling} providing
flatness of the filling.

\begin{proof}[Proof of Theorem~\ref{main:boundary}]
\emph{Part \textup{(i)}.}
By~\eqref{bnd:H-upper}, $0<H_g\leq h$. Dropping the nonnegative
$|\bar\nabla u|^2$ term in~\eqref{bnd:remainder} and applying the
arithmetic--geometric mean inequality give
\begin{align*}
 \mathcal R_h
 &\geq
 \frac{n-1}{12H_g}(H_g-h)^2
 \bigl[
 u(H_g^3+2hH_g^2+4h^2H_g+2h^3)
 +u^{-1}(H_g+2h)
 \bigr]\\
 &\geq
 \frac{n-1}{6H_g}(H_g-h)^2
 \sqrt{
 (H_g^3+2hH_g^2+4h^2H_g+2h^3)(H_g+2h)
 }.
\end{align*}
Moreover,
\[
\begin{split}
 &(H_g^3+2hH_g^2+4h^2H_g+2h^3)(H_g+2h)
 -27h^2H_g^2\\
 &\qquad
 =(h-H_g)
 \bigl(4h^3+14h^2H_g-5hH_g^2-H_g^3\bigr)\\
 &\qquad
 =(h-H_g)
 \bigl(4h^3+8h^2H_g
       +H_g(h-H_g)(6h+H_g)\bigr)
 \geq0.
\end{split}
\]
Hence
\[
 \mathcal R_h
 \geq
 \frac{\sqrt3(n-1)}2h(H_g-h)^2
 =
 \frac{3\sqrt3}{2}\Lambda
 \left(1-\frac{H_g}{h}\right)^2,
\]
where $\Lambda=(n-1)h^3/3$.

Since $H_2(g)=\Lambda$ is constant, $h$ is constant, and the
reference-derivative term in~\eqref{bnd:integrated-identity}
vanishes. Therefore,
\[
 \frac{3\sqrt3}{2}\Lambda
 \int_\Sigma
 \left(1-\frac{H_g}{h}\right)^2\,\mathrm d\sigma_g
 \leq
 2\int_M\chi\sigma_2(A_g)\,\dd v_g.
\]
This proves the first inequality in~\eqref{bnd:defect-estimate}.
Finally, since
$A_g\in\overline{\Gamma_2^+}$ implies $\sigma_2(A_g)\geq0$,
\[
 \int_M\chi\sigma_2(A_g)\,\dd v_g
 \leq
 \|\chi\|_{L^\infty(M)}
 \int_M\sigma_2(A_g)\,\dd v_g,
\]
which proves the second inequality.

\emph{Part \textup{(ii)}.}
Suppose that $\sigma_2(A_g)=0$. Part~\textup{(i)} gives
\[
 \Lambda\int_\Sigma\left(1-\frac{H_g}{h}\right)^2
 \,\mathrm d\sigma_g=0.
\]
Since $\Lambda>0$ and the integrand is continuous and nonnegative,
$H_g=h$ on $\Sigma$.
It follows from \eqref{bnd:H2} that $\sigma_1(B)=0$. Moreover,
\eqref{bnd:CK-flux} gives
\[
 0=\int_\Sigma u\sigma_2(B)\,\mathrm d\sigma_g
   =-\frac12\int_\Sigma u|B|^2\,\mathrm d\sigma_g,
\]
where the second equality uses
$2\sigma_2(B)=\sigma_1(B)^2-|B|^2$. Thus $B=0$. Finally,
$A_g\in\overline{\Gamma_2^+}$ implies $R_g\geq0$, and
Lemma~\ref{bnd:flat-filling} shows that $(M,g)$ is isometric to a
Euclidean ball.
\end{proof}

\begin{samepage}
\begin{remark}[Nonincreasing boundary data]
\label{bnd:monotone-extension}
\label{bnd:monotone}
\label{bnd:power-law}
\label{bnd:CW-comparison}
Under the other hypotheses of Theorem~\ref{main:boundary}, the same
method applies to $H_2(g)=F(u)>0$, where $F$ is $C^1$ and
nonincreasing on an interval containing the boundary range of $u$.
Set $h(u)=(3F(u)/(n-1))^{1/3}$. The reference-derivative term becomes
\[
 h\langle\bar\nabla h,\bar\nabla u\rangle
 =\frac{F'(u)}{(n-1)h}|\bar\nabla u|^2\leq0.
\]
If $\sigma_2(A_g)=0$, formulas~\eqref{bnd:integrated-identity}
and~\eqref{bnd:remainder} force $H_g=h$ and the integral of this
derivative term to vanish. The same flux argument gives $B=0$,
and Lemma~\ref{bnd:flat-filling} gives Euclidean-ball rigidity.
If $F'<0$ on the boundary range of $u$, the vanishing derivative
integral also gives $\bar\nabla u=0$; since $\Sigma$ is connected,
$u|_\Sigma$ is constant.
\end{remark}
\end{samepage}

\section{Rigidity in a positive Einstein conformal class}
\label{clo:section}

Throughout this section, \((M^n,g_0)\) is closed and connected,
\(n\ge3\), and
\[
\Ric_{g_0}=\frac{R_{g_0}}n g_0,\qquad R_{g_0}>0.
\]
We write \(g=u^2g_0\), where \(u\) is smooth and positive.
Differential operators are displayed with their metric subscripts.
Unless otherwise indicated, tensor norms and integrals use \(g\).
The proofs first establish positivity from the curvature equation,
then choose a reference in \eqref{id:variable-formula} that makes
the integral nonnegative. The key algebraic relation is \eqref{id:sigma}.
\begin{lemma}[Positivity of the scalar curvature]\label{clo:positive-branch}
At any global maximum point of \(u\) on \(M\), the tensor \(A_g\)
is positive definite. Consequently, if \(\sigma_2(A_g)>0\) everywhere,
then \(R_g>0\) everywhere.
\end{lemma}
\begin{proof}
Since $M$ is closed, $u$ attains its maximum at some point
$x_0\in M$. At $x_0$,
\[
 \dd\log u=0,
 \qquad
 \nabla_{g_0}^2\log u\leq0.
\]
The conformal transformation law~\eqref{id:conformal-schouten}
therefore gives
\[
 A_g(x_0)
 =
 A_{g_0}(x_0)-\nabla_{g_0}^2\log u(x_0)
 \geq
 \frac{R_{g_0}}{2n(n-1)}g_0>0.
\]
Hence $\sigma_1(A_g)(x_0)>0$.

Now suppose that $\sigma_2(A_g)>0$ on $M$. If
$\sigma_1(A_g)$ vanished at some point $x\in M$, then
\[
 2\sigma_2(A_g)(x)
 =
 \sigma_1(A_g)(x)^2-|A_g(x)|_g^2
 =
 -|A_g(x)|_g^2
 \leq0,
\]
a contradiction. Thus $\sigma_1(A_g)$ is nowhere zero. Since
$M$ is connected and $\sigma_1(A_g)(x_0)>0$, it follows that
\[
 \sigma_1(A_g)>0
 \qquad\text{on }M.
\]
Finally, $R_g=2(n-1)\sigma_1(A_g)$, and hence $R_g>0$ everywhere.
\end{proof}

We use the classical Obata identity in the normalization
of~\cite{LiWei25}, also recorded in
\cite[Equation~(2.17)]{CWSharp}. It follows for every $n\ge3$
by integrating \eqref{id:tensor-product} with $T=E_g$ and
$f=u$, and using \eqref{id:bianchi-tracefree} and \eqref{id:hess}.

\begin{lemma}[Classical Obata identity]\label{clo:pairing}
For every smooth positive \(u\),
\begin{equation}\label{clo:pairing-formula}
\int_M\langle\nabla_g R_g,\nabla_g u\rangle_g\,\dd v_g
=\frac{2n}{(n-2)^2}\int_Mu|E_g|^2\,\dd v_g.
\end{equation}
Consequently, if \(R_g\) is constant, then \(g\) is Einstein.
\end{lemma}

If $g$ is Einstein, Obata's conformal classification
\cite{Obata62,Obata71} implies that $u$ is constant unless $(M,g_0)$
is a round sphere up to scaling. In the spherical case the Einstein
metrics in the conformal class are
\begin{equation}\label{clo:obata-family}
 g=s^2\Phi^*g_0,\qquad s>0,
\end{equation}
where $\Phi$ is a conformal diffeomorphism. In particular, on a
nontrivial spherical space form the conformal factor is constant.

\subsection{Constant \texorpdfstring{\(\sigma_2\)}{sigma-two} curvature}

\begin{proof}[Proof of Theorem~\ref{main:constant}]
Let $x_0\in M$ be a global maximum point of $u$. By
Lemma~\ref{clo:positive-branch}, $A_g(x_0)$ is positive definite.
Hence all its eigenvalues are positive, and therefore
\[
 \Lambda=\sigma_2(A_g)(x_0)>0.
\]
The second conclusion of
Lemma~\ref{clo:positive-branch} then gives
\[
 R_g>0
 \qquad\text{on }M.
\]
Set
\[
 r\coloneqq\sqrt{8n(n-1)\Lambda}>0.
\]
The Schouten decomposition gives
\[
 R_g^2-r^2=\frac{4n(n-1)}{(n-2)^2}|E_g|^2\ge0.
\]
Since $R_g$ and $r$ are positive, $R_g\ge r$.
In~\eqref{id:variable-formula}, $\nabla_g r=0$, and the choice
$r^2=8n(n-1)\Lambda$ cancels the term
$u(r^2/(4n)-2(n-1)\sigma_2(A_g))$.
Integration over the closed manifold therefore gives
\begin{equation}\label{clo:constant-reference-integral}
 0=\int_M\frac{R_g-r}{4n}
 \left(n(n-1)u^{-1}|\nabla_g u|^2+ur+u^{-1}R_{g_0}\right)\dd v_g.
\end{equation}
The integrand is continuous and nonnegative. The coefficient of
$R_g-r$ is strictly positive because $u$, $r$, and $R_{g_0}$
are positive. Hence $R_g=r$ everywhere, and the displayed
Schouten relation gives $E_g=0$. Thus $g$ is Einstein.
If $(M,g_0)$ is not a round sphere up to scaling, Obata's
classification implies that $u$ is constant.
\end{proof}

\begin{samepage}
\begin{remark}[Nonincreasing prescribed curvature]
\label{clo:monotone-extension}\label{main:monotone}
\label{clo:monotone}\label{clo:powers}
Under the hypotheses of Theorem~\ref{main:constant} on $(M,g_0)$ and $u$,
the same method applies to $\sigma_2(A_g)=F(u)$ when $F$ is $C^1$
and nonincreasing on an interval containing the range of $u$.
At a maximum of $u$, $F(\max_Mu)>0$; monotonicity and
Lemma~\ref{clo:positive-branch} then give $F(u)>0$ and $R_g>0$.
Set $r(u)=\sqrt{8n(n-1)F(u)}$. The same Schouten decomposition gives
$R_g\ge r(u)$. Formula~\eqref{clo:constant-reference-integral}
holds with $r$ replaced by $r(u)$ and the additional integrand
\[
 -\frac{n-1}{2n}r'(u)|\nabla_g u|^2\ge0,
 \qquad r'(u)=\frac{4n(n-1)F'(u)}{r(u)}\le0.
\]
Both summands are nonnegative, and the coefficient of $R_g-r(u)$
is strictly positive. Hence $R_g=r(u)$ and $E_g=0$ as above.
Thus the metric $g$ is Einstein, and
$F(u)=\sigma_2(A_g)$ is constant. Theorem~\ref{main:constant}
therefore gives the same classification as in the constant-curvature
case. If $F$ is strictly decreasing on the range of $u$, then the
constancy of $F(u)$ forces $u$ to be constant, even when $g_0$ is
round.
\end{remark}
\end{samepage}

\subsection{A constant curvature quotient}

\begin{proof}[Proof of Theorem~\ref{main:quotient}\textup{(i)}]

At a global maximum point of \(u\) on \(M\), Lemma~\ref{clo:positive-branch}
gives both \(\sigma_2(A_g)>0\) and \(\sigma_1(A_g)>0\), so \(\Lambda>0\).
The equation implies
\[
|A_g|^2=\sigma_1(A_g)(\sigma_1(A_g)-2\Lambda)\ge0.
\]
Thus \(\sigma_1(A_g)\) takes values in \((-\infty,0]\cup[2\Lambda,\infty)\).
Connectedness and the positive value of \(\sigma_1(A_g)\) at that
maximum point of \(u\) force \(\sigma_1(A_g)\ge2\Lambda>0\) everywhere.
Thus $R_g>0$ and $\sigma_2(A_g)=\Lambda\sigma_1(A_g)>0$;
in particular, the solution lies in $\Gamma_2^+$.
Since $2(n-1)\sigma_2(A_g)=\Lambda R_g$, the Schouten
decomposition~\eqref{id:sigma} gives pointwise
\[
 R_g(R_g-4n\Lambda)
 =\frac{4n(n-1)}{(n-2)^2}|E_g|^2\ge0.
\]
As $R_g>0$, it follows that $R_g\ge4n\Lambda$.
Choose the constant reference $r=4n\Lambda$ in
\eqref{id:variable-formula}. Its gradient term vanishes, and the
scalar terms cancel:
\[
 \frac{ur(R_g-r)}{4n}
 +u\left(\frac{r^2}{4n}-\Lambda R_g\right)
 =uR_g\left(\frac r{4n}-\Lambda\right)=0.
\]
Integrating the divergence over the closed manifold gives
\[
 0=\int_M\frac{R_g-4n\Lambda}{4n}
 \left(n(n-1)u^{-1}|\nabla_g u|^2+u^{-1}R_{g_0}\right)\dd v_g.
\]
The integrand is continuous and nonnegative, and the expression in
parentheses is strictly positive because $u>0$ and $R_{g_0}>0$.
Thus $R_g=4n\Lambda$ everywhere, and the displayed Schouten
relation gives $E_g=0$.
Obata's theorem gives $g=s^2g_0$, since a background which is not
locally conformally flat cannot be round. Finally,
$R_g=s^{-2}R_{g_0}$ gives $s^2=R_{g_0}/(4n\Lambda)$,
as asserted.
\end{proof}

\section{Sharp comparison in a positive Einstein conformal class}
\label{sharp:section}
We prove Theorem~\ref{main:quotient}\textup{(ii)} by applying
part~\textup{(i)} to a minimizer of the quotient studied by Ge and G.~Wang.
Throughout this section, $(M^n,g_0)$ is closed and connected,
$n\ge5$, and $g_0$ is Einstein with scalar curvature $R_{g_0}>0$.
All conformal metrics below are smooth. We write
\[
 \mathcal C_k(M)\coloneqq \{g\in[g_0]:A_g\in\Gamma_k^+
                       \text{ on }M\}.
\]
Thus $\mathcal C_k(M)$ is the class of smooth conformal metrics satisfying
the open cone condition everywhere. In particular,
$\mathcal C_1(M)\coloneqq \{g\in[g_0]:R_g>0\}$.
Using \eqref{intro:normalized-functionals}, define the
scale-invariant quotient
\begin{equation}\label{sharp:functionals}
 \mathcal Q_{2,1}(g)\coloneqq
 \frac{\displaystyle\int_M\sigma_2(A_g)\,\dd v_g}
      {\left(\displaystyle\int_M\sigma_1(A_g)\,\dd v_g\right)^{(n-4)/(n-2)}} = \frac{\mathcal F_2(g)} {\left(\mathcal F_1(g)\right)^{\frac{n - 4}{n-2}}}.
\end{equation}
It is defined on $\mathcal C_1(M)$, where
$\int_M\sigma_1(A_g)\,\dd v_g>0$.
The usual Yamabe functional is $2(n-1)\mathcal F_1$.

On the round sphere, every $g\in\mathcal C_1(\mathbb S^n)$ satisfies
the known comparison
\begin{equation}
 \frac{\mathcal F_2(g)}{\mathcal F_2(g_0)}
 \geq
 \left(
 \frac{\mathcal F_1(g)}{\mathcal F_1(g_0)}
 \right)^{\frac{n-4}{n-2}}
 \geq1;
 \label{sharp:spherical-known}
\end{equation}
see~\cite[(2.1)]{FP24}. Since the normalized ratios are unchanged
under finite coverings, the same comparison holds on every spherical
space form. Equality yields an Einstein metric; hence
$g=s^2g_0$ on a nontrivial spherical space form, while on the sphere
the equality metrics are $g=s^2\Phi^*g_0$ for conformal
diffeomorphisms $\Phi$.

We use the following existence result of
Ge and G.~Wang \cite[Proposition~1 and Theorem~1]{GW13}: if $n>4$ and
$\mathcal C_2(M)\ne\varnothing$, then
\begin{equation}
 0<\inf_{\mathcal C_1(M)}\mathcal Q_{2,1}
   =\inf_{\mathcal C_2(M)}\mathcal Q_{2,1}<\infty,
 \label{sharp:GW}
\end{equation}
and the infimum is attained by a conformal metric in $\mathcal C_2(M)$.
This result does not require local conformal flatness.
Their Proposition~1 supplies the required positivity, and evaluation
at a metric in $\mathcal C_2(M)$ gives finiteness. The $C^{2,\gamma}$
minimizer constructed in~\cite[Section~5]{GW13} is smooth by elliptic
regularity: its constant quotient equation is uniformly elliptic
along this admissible metric on the compact manifold.

We also recall the quotient Euler--Lagrange equation. It follows from
the first variations in~\cite[Section~2, Lemma~1 and (2.9)]{GW13}.

\begin{lemma}[Quotient Euler--Lagrange equation, Ge and G.~Wang \cite{GW13}]
\label{sharp:variation}
Every smooth minimizer of $\mathcal Q_{2,1}$ in $\mathcal C_1(M)$ satisfies
\begin{equation}
 \sigma_2(A_g)=
 \frac{\displaystyle\int_M\sigma_2(A_g)\,\dd v_g}
      {\displaystyle\int_M\sigma_1(A_g)\,\dd v_g}\sigma_1(A_g).
 \label{sharp:quotientEL}
\end{equation}
\end{lemma}
\begin{proof}
Set $I_k(g)=\int_M\sigma_k(A_g)\,\dd v_g$ for $k=1,2$.
For $g_t=e^{2t\varphi}g$, the first variations give
\[
 \left.\frac{\dd}{\dd t}I_k(g_t)\right|_{t=0}
 =(n-2k)\int_M\varphi\sigma_k(A_g)\,\dd v_g,
 \qquad k=1,2.
\]
The $k=2$ formula uses only $\operatorname{div}_gT_1(A_g)=0$,
so both formulas hold on general closed manifolds. Hence
\[
 \left.\frac{\dd}{\dd t}\mathcal Q_{2,1}(g_t)\right|_{t=0}
 =(n-4)I_1(g)^{-\frac{n-4}{n-2}}\int_M\varphi
 \left(\sigma_2(A_g)-\frac{I_2(g)}{I_1(g)}\sigma_1(A_g)\right)\dd v_g.
\]
For each fixed smooth $\varphi$, the path remains in $\mathcal C_1(M)$
for both signs of sufficiently small $t$: $R_g$ has a positive
minimum on the compact manifold, and scalar curvature depends
continuously on the metric in $C^2$. The minimizer is therefore
stationary under every such variation. Since $n>4$ and $\varphi$
is arbitrary, this proves \eqref{sharp:quotientEL} without local
conformal flatness.
\end{proof}

The remaining step is to identify the minimizer constructed by Ge and G.~Wang on an
Einstein background which is not locally conformally flat. This is
precisely where Theorem~\ref{main:quotient}\textup{(i)} enters.

\begin{proof}[Proof of Theorem~\ref{main:quotient}\textup{(ii)}]
Since $g_0$ is positive Einstein, $g_0\in\mathcal C_2(M)$.
By~\eqref{sharp:GW}, the quotient $\mathcal Q_{2,1}$ admits a
smooth minimizer $\widehat g\in\mathcal C_2(M)$. Its
Euler--Lagrange equation, given by
Lemma~\ref{sharp:variation}, is
\[
 \sigma_2(A_{\widehat g})
 =
 \Lambda\sigma_1(A_{\widehat g})
\]
for some $\Lambda>0$. Part~\textup{(i)} therefore gives
$\widehat g=s^2g_0$. Since $\mathcal Q_{2,1}$ is invariant under
constant rescaling,
\[
 \mathcal Q_{2,1}(g)
 \geq
 \mathcal Q_{2,1}(g_0)
 \qquad
 \text{for every }g\in\mathcal C_1(M).
\]
Using
\[
 \mathcal Q_{2,1}
 =
 \frac{\mathcal F_2}
 {\mathcal F_1^{(n-4)/(n-2)}},
\]
this is equivalent to
\[
 \frac{\mathcal F_2(g)}{\mathcal F_2(g_0)}
 \geq
 \left(
 \frac{\mathcal F_1(g)}{\mathcal F_1(g_0)}
 \right)^{\frac{n-4}{n-2}}.
\]

Next, let $g_Y\in[g_0]$ be a Yamabe minimizer
~\cite{LeeParker87}. Since $R_{g_Y}$ is constant,
Lemma~\ref{clo:pairing} shows that $g_Y$ is Einstein. Obata's
classification~\cite{Obata62,Obata71} then gives
$g_Y=s^2g_0$. Hence $g_0$ is a Yamabe minimizer in its conformal
class, and therefore
\[
 \mathcal F_1(g)\geq\mathcal F_1(g_0).
\]
This proves~\eqref{sharp:strong}.

Since
\[
 \mathcal F_2(g_0)
 =
 \frac{R_{g_0}^2}{8n(n-1)}
 \Vol_{g_0}(M)^{4/n},
\]
the first and last terms in~\eqref{sharp:strong} also give
\eqref{sharp:pure}.

If equality holds in the first comparison, then $g$ is a minimizer
of $\mathcal Q_{2,1}$; Lemma~\ref{sharp:variation} and
part~\textup{(i)} give $g=s^2g_0$. If equality holds in the second
comparison, then $g$ is a Yamabe minimizer, and the preceding
argument again gives $g=s^2g_0$. Equality in~\eqref{sharp:pure}
forces equality throughout~\eqref{sharp:strong}. Conversely,
every constant rescaling of $g_0$ gives equality.
\end{proof}

\begin{remark}[Two variational routes]
\label{sharp:scope}
For $n\geq5$, Ge, G.~Wang, and W.~Wei~\cite[Theorem~1.1]{GWW26}
show that, whenever $\mathcal C_2(M)\ne\varnothing$, the infimum
of $\mathcal F_2$ over $\mathcal C_1(M)$ is positive and is attained
by a metric in $\mathcal C_2(M)$. Combined with
Theorem~\ref{main:constant}, this gives an alternative proof of
the pure comparison~\eqref{sharp:pure}.

To obtain the stronger comparison~\eqref{sharp:strong}, we instead
use the minimizer theorem of Ge and G.~Wang for the quotient
$\mathcal Q_{2,1}$~\cite{GW13}, together with
Theorem~\ref{main:quotient}\textup{(i)}. The intermediate
$\mathcal F_1$ term in~\eqref{sharp:strong} will be needed in the
stability argument below. Both minimization problems are posed in
the positive-scalar-curvature class $\mathcal C_1(M)$; without this
restriction, the pure $\sigma_2$ minimization problem is generally
unbounded below. In dimension four, the total $\sigma_2$ curvature
is conformally invariant.
\end{remark}

\section{Global stability on a fixed nonround Einstein manifold}
\label{stab:section}

Let $(M^n,g_0)$ be closed and connected, with $n\geq5$ and
\[
 \Ric_{g_0}=(n-1)\kappa g_0,
 \qquad
 \kappa>0,
 \qquad
 \Vol_{g_0}(M)=1,
\]
and assume that $(M,g_0)$ is not isometric, up to scaling, to the
round sphere. Unless another metric is indicated, all differential
operators, integrals, and norms in this section are taken with
respect to $g_0$.

Let
\begin{equation}
 \label{stab:normalization}
 g=z^{8/(n-4)}g_0,
 \qquad
 z>0,
 \qquad
 \int_Mz^{4n/(n-4)}\,\dd v_{g_0}=1.
\end{equation}
Then $\Vol_g(M)=1$, and
\[
 \mathcal F_2(g_0)
 =
 \sigma_2(A_{g_0})
 =
 \frac{n(n-1)\kappa^2}{8}.
\]
For a real-valued function $h$, set
\[
 \|h\|_{H^1}^2
 \coloneqq
 \int_M(h^2+|\nabla_{g_0}h|^2)\,\dd v_{g_0},
 \qquad
 \|h\|_{W^{1,4}}^4
 \coloneqq
 \int_M(|h|^4+|\nabla_{g_0}h|^4)\,\dd v_{g_0}.
\]

The proof combines the sharp comparison of
Theorem~\ref{main:quotient} with quantitative Yamabe stability and
Case's positive decomposition of the $\sigma_2$ energy. For a small
$\sigma_2$ deficit, these give the required $H^1$ and $W^{1,4}$
estimates. The complementary large-deficit range follows directly
from the energy identity and the volume normalization. We first
record the two analytic inputs.

The first is a consequence of
\cite[Proposition~3.1 and Lemma~4.1]{ENS22} in the normalization
used below.

\begin{lemma}[Yamabe stability, Engelstein, Neumayer, and Spolaor \cite{ENS22}]
\label{stab:yamabe}
Let
\[
 \mathcal Y(w)\coloneqq \int_M\left(
       \frac{4(n-1)}{n-2}|\nabla_{g_0} w|^2+R_{g_0}w^2\right)\dd v_{g_0},
 \qquad R_{g_0}=n(n-1)\kappa.
\]
The function $w$ is the Yamabe conformal factor in
$g=w^{4/(n-2)}g_0$; under unit volume, $\mathcal Y(w)$ is the total
scalar curvature of this metric. There are an energy threshold
$\delta_Y>0$ and a stability constant $C_Y<\infty$ such that, whenever $w>0$ is
smooth, $\int_Mw^{2n/(n-2)}\,\dd v_{g_0}=1$, and
$\mathcal Y(w)-R_{g_0}\le\delta_Y$, one has
\begin{equation}
 \|w-1\|_{H^1}^2\le C_Y\bigl(\mathcal Y(w)-R_{g_0}\bigr).
 \label{stab:yamabeestimate}
\end{equation}
\end{lemma}
\begin{proof}[Application of the cited results]
By the Yamabe theorem and Obata's theorem, $1$ is the unique
positive normalized Yamabe minimizer and
$\mathcal Y(1)=R_{g_0}$~\cite{LeeParker87,Obata62,Obata71}.
Since $g_0$ is Einstein and nonround, the strict
Lichnerowicz--Obata inequality
\[
 \lambda_1(-\Delta_{g_0})>n\kappa
\]
makes the constrained Yamabe Hessian at $1$ nondegenerate.
Hence~\cite[Proposition~3.1]{ENS22} gives the quadratic estimate
\eqref{stab:yamabeestimate} in an $H^1$ neighborhood of $1$.
The normalization factor in their distance is bounded on this
neighborhood and is absorbed into the constant.
Finally, the compactness result~\cite[Lemma~4.1]{ENS22}, together
with uniqueness of the normalized minimizer, shows that every
normalized function with sufficiently small Yamabe deficit lies in
this neighborhood. Shrinking the deficit threshold if necessary
proves the lemma.
\end{proof}

The second input is the Einstein specialization of Case's energy
identity~\cite[Lemma~2.1, Remark~2.2, and proof of
Corollary~2.3]{Case21}. The scalar-curvature formula below follows
directly from~\eqref{id:conformal-trace}.

\begin{lemma}[Energy identity, Case \cite{Case21}]
\label{stab:energy}
Let $z$ be smooth and positive, and let
$g=z^{8/(n-4)}g_0$. Then
\begin{equation}
 z^{2n/(n-4)}\sigma_1(A_g)
 =-\frac{2}{n-4}\Delta_{g_0}(z^2)
   -\frac{16}{(n-4)^2}|\nabla_{g_0}z|^2
   +\frac{n\kappa}{2}z^2.
 \label{stab:K}
\end{equation}
Moreover,
\begin{equation}
\begin{split}
 \int_M\sigma_2(A_g)\,\dd v_g
 =\int_M\biggl[&
 \sigma_2(A_{g_0})z^4
 +\frac{4}{n-4}z^{2n/(n-4)}
      \sigma_1(A_g)|\nabla_{g_0}z|^2\\
 &+\frac{32}{(n-4)^3}|\nabla_{g_0}z|^4
 +\frac{2(n-2)\kappa}{n-4}
      z^2|\nabla_{g_0}z|^2
 \biggr]\,\dd v_{g_0}.
\end{split}
\label{stab:positiveenergy}
\end{equation}
If $R_g>0$, every term on the right-hand side of
\eqref{stab:positiveenergy} is nonnegative.
\end{lemma}

\begin{proof}[Proof of Theorem~\ref{stab:global}]
Write $\mathcal D(g)\coloneqq\mathcal F_2(g)-\mathcal F_2(g_0)$ for the
energy deficit. All constants depend only on the fixed background.
Integrals without a displayed measure are taken with respect to
$\dd v_{g_0}$, and $\|\cdot\|_p$ denotes the $L^p(g_0)$ norm.

\paragraph{\textbf{Step 1. Control of the Yamabe factor.}}
Use the Yamabe conformal factor
\[
 w\coloneqq z^{2(n-2)/(n-4)}.
\]
Then $\int_M w^{2n/(n-2)}\,\dd v_{g_0}=1$ and $g=w^{4/(n-2)}g_0$. In particular,
$\mathcal Y(w)=2(n-1)\mathcal F_1(g)$.
By Theorem~\ref{main:quotient}\textup{(ii)} and the extension to
spherical space forms discussed after~\eqref{sharp:spherical-known},
we have
\[
 1\le\frac{\mathcal F_1(g)}{\mathcal F_1(g_0)}
 \le\left(\frac{\mathcal F_2(g)}{\mathcal F_2(g_0)}\right)^{\frac{n-2}{n-4}}.
\]
Since $\mathcal F_2(g_0)>0$, the comparison gives $\mathcal D(g)\ge0$.
For $\mathcal D(g)\le1$, use $\mathcal Y(1)=R_{g_0}$ to obtain
\[
 0\le\mathcal Y(w)-R_{g_0}
 \le R_{g_0}\left[
 \left(1+\frac{\mathcal D(g)}{\mathcal F_2(g_0)}\right)^{(n-2)/(n-4)}-1\right]
 \le C_0\mathcal D(g).
\]
The last bound follows from the mean value theorem on the fixed
interval $[1,1+\mathcal F_2(g_0)^{-1}]$. Choose
$\delta_{\mathrm{stab}}=\min\{1,\delta_Y/C_0\}$.
Lemma~\ref{stab:yamabe} then gives
\begin{equation}
\begin{gathered}
 \mathcal D(g)\le\delta_{\mathrm{stab}}\\
 \Longrightarrow\quad
 \|w-1\|_{H^1}^2\le C\mathcal D(g).
\end{gathered}
 \label{stab:wcontrol}
\end{equation}

\paragraph{\textbf{Step 2. Control of the gradient in $L^4$.}}
Assume $\mathcal D(g)\le\delta_{\mathrm{stab}}$.
Since $\Vol_g(M)=1$, the left-hand side of the energy identity \eqref{stab:positiveenergy} is
$\mathcal F_2(g)$. Discarding its other nonnegative gradient terms,
using $\sigma_1(A_g)>0$, gives
\begin{equation}
 \frac{32}{(n-4)^3}\int_M|\nabla_{g_0} z|^4\,\dd v_{g_0}
 \le \mathcal D(g)+\mathcal F_2(g_0)\left(1-\int_Mz^4\,\dd v_{g_0}\right).
 \label{stab:quarticpre}
\end{equation}
We must therefore control the loss in the $L^4$ mass of $z$.
Since $z^4=w^{2(n-4)/(n-2)}$, we use the volume normalization to add
a term with zero integral:
\[
 1-\int_Mz^4
 =\int_M\left[
 1-w^{2(n-4)/(n-2)}
 +\frac{n-4}{n}\bigl(w^{2n/(n-2)}-1\bigr)\right].
\]
For every $t\ge0$,
\begin{equation}
\begin{split}
 0\le{}&1-t^{2(n-4)/(n-2)}
       +\frac{n-4}{n}\bigl(t^{2n/(n-2)}-1\bigr)\\
 \le{}&C\bigl((t-1)^2+|t-1|^{2n/(n-2)}\bigr).
\end{split}
 \label{stab:scalarremainder}
\end{equation}
The lower bound follows from the concavity of
$s^{(n-4)/n}$ and its tangent-line inequality at $s=1$, applied with
$s=t^{2n/(n-2)}$. For the upper bound, the expression and its first
derivative vanish at $t=1$, so Taylor's theorem gives a bound by
$C(t-1)^2$ when $t$ is close to $1$. Away from $1$, the expression is
bounded for $0\leq t\leq1/2$, while for large $t$ it is bounded by
$Ct^{2n/(n-2)}$. These estimates yield
\eqref{stab:scalarremainder}.

The compact-manifold Sobolev and Poincar\'e--Sobolev inequalities
used here and below take the following form~\cite[Chapter~3]{Hebey96}:
\begin{equation}\label{stab:sobolev-tools}
 \|h\|_{2n/(n-2)}\le C\|h\|_{H^1},\qquad
 \left\|h-\int_Mh\,\dd v_{g_0}\right\|_{4n/(n-4)}
       \le C\|\nabla_{g_0} h\|_4.
\end{equation}
They apply to every smooth $h$ on the fixed connected background;
the average has this form because $\Vol_{g_0}(M)=1$.
By \eqref{stab:wcontrol} and the first inequality in
\eqref{stab:sobolev-tools},
\begin{align*}
 0\le1-\int_Mz^4
 &\le C\left(\|w-1\|_2^2+
             \|w-1\|_{2n/(n-2)}^{2n/(n-2)}\right)\\
 &\le C\Bigl(\mathcal D(g)
       +\mathcal D(g)^{n/(n-2)}\Bigr)\\
 &\le C\mathcal D(g).
\end{align*}
Here the last inequality uses $\mathcal D(g)\le1$ and
$n/(n-2)>1$. Substitution into \eqref{stab:quarticpre} yields
\begin{equation}
 \int_M|\nabla_{g_0} z|^4\,\dd v_{g_0}\le C\mathcal D(g).
 \label{stab:gradientfour}
\end{equation}

\paragraph{\textbf{Step 3. Control of both Sobolev distances.}}
The relation $w=z^{2(n-2)/(n-4)}$ gives
$|z-1|\le|w-1|$: on $0<z\le1$ one has $w\le z$, and on $z\ge1$
one has $w\ge z$. Hence $\|z-1\|_2^2\le C\mathcal D(g)$.
To control the $L^2$ norm of $\nabla z$, split
\[
 M=\{z\geq1/2\}\cup\{z<1/2\}.
\]
On $\{z\geq1/2\}$, the relation
\[
 \nabla z
 =
 \frac{n-4}{2(n-2)}
 z^{-n/(n-4)}\nabla w
\]
gives $|\nabla z|\leq C|\nabla w|$, and hence
\[
 \int_{\{z\geq1/2\}}|\nabla z|^2
 \leq C\mathcal D(g).
\]
On $E\coloneqq\{z<1/2\}$, one has
\[
 |w-1|
 \geq
 1-2^{-2(n-2)/(n-4)}
 =:c_n>0.
\]
Therefore,
\[
 \Vol_{g_0}(E)
 \leq
 c_n^{-2}\int_E|w-1|^2\,\dd v_{g_0}
 \leq C\mathcal D(g).
\]
Using H\"older's inequality and~\eqref{stab:gradientfour},
\[
 \int_E|\nabla z|^2\,\dd v_{g_0}
 \leq
 \Vol_{g_0}(E)^{1/2}
 \left(\int_E|\nabla z|^4\,\dd v_{g_0}\right)^{1/2}
 \leq C\mathcal D(g).
\]
Combining the two regions yields
\[
 \int_M|\nabla z|^2\,\dd v_{g_0}
 \leq C\mathcal D(g).
\]
Together with $\|z-1\|_2^2\leq\|w-1\|_2^2$, this proves
\[
 \|z-1\|_{H^1}^2\leq C\mathcal D(g).
\]

It remains to control the zeroth-order term in the
$W^{1,4}$ norm. Since $\Vol_{g_0}(M)=1$, the critical
Poincar\'e--Sobolev inequality gives
\[
 \left\|z-\int_Mz\,\dd v_{g_0}\right\|_{4n/(n-4)}
 \leq C\|\nabla_{g_0}z\|_4.
\]
The normalization in~\eqref{stab:normalization} implies
$\|z\|_{4n/(n-4)}=1$. Hence the reverse triangle inequality yields
\[
\begin{aligned}
 \left|\int_Mz\,\dd v_{g_0}-1\right|
 &=\left|\left\|\int_Mz\,\dd v_{g_0}\right\|_{4n/(n-4)}
       -\|z\|_{4n/(n-4)}\right|\\
 &\leq\left\|z-\int_Mz\,\dd v_{g_0}\right\|_{4n/(n-4)}
 \leq C\|\nabla_{g_0}z\|_4.
\end{aligned}
\]
Since $4n/(n-4)>4$ and $\Vol_{g_0}(M)=1$, it follows that
\[
 \|z-1\|_4
 \leq C\|\nabla_{g_0}z\|_4.
\]
Together with~\eqref{stab:gradientfour}, this gives
\[
 \|z-1\|_{W^{1,4}}^4
 \leq C\mathcal D(g).
\]
Combining this with the preceding $H^1$ estimate proves
\eqref{stab:estimate} in the small-deficit case.

\paragraph{\textbf{Step 4. Large deficits.}}
Suppose $\mathcal D(g)\ge\delta_{\mathrm{stab}}$.
H\"older's inequality and the normalization give $\int z^4\le1$.
Equation~\eqref{stab:positiveenergy} directly gives
$\int|\nabla_{g_0} z|^4\le C\mathcal F_2(g)$.
The elementary bound $|z-1|^4\le8(z^4+1)$ controls the zeroth-order
$W^{1,4}$ term. Also $\int z^2\le1$, so $\|z-1\|_2^2\le4$, and
$\|\nabla_{g_0} z\|_2^2\le\|\nabla_{g_0} z\|_4^2$. Therefore
\[
 \|z-1\|_{W^{1,4}}^4\le C\bigl(1+\mathcal F_2(g)\bigr),\qquad
 \|z-1\|_{H^1}^2\le C\bigl(1+\sqrt{\mathcal F_2(g)}\bigr).
\]
Since $\mathcal F_2(g)\ge\mathcal F_2(g_0)>0$, both right-hand sides
are bounded by $C\mathcal F_2(g)$. The fixed threshold gives
\[
 \mathcal F_2(g)
 \le\left(1+\frac{\mathcal F_2(g_0)}{\delta_{\mathrm{stab}}}\right)
       \mathcal D(g).
\]
Combining the small- and large-deficit estimates, there exists a
constant $C=C(M,g_0)$ such that
\[
 \|z-1\|_{H^1}^2+\|z-1\|_{W^{1,4}}^4
 \leq C\mathcal D(g)
\]
for every normalized conformal metric $g$ with $R_g>0$.
Taking $c_*=C^{-1}$ proves
\eqref{stab:estimate}.

\end{proof}

\begin{remark}[The spherical case]
\label{stab:exponents}
On the round sphere, the normalized minimizer is not unique:
conformal pullbacks of the round metric are also minimizers.
Accordingly, quantitative stability must measure the distance to
the full conformal family, rather than to the single factor $1$;
see Frank and Peteranderl~\cite{FP24}. For a fixed nonround positive
Einstein background, the normalized minimizer is unique, which
allows the distance in Theorem~\ref{stab:global} to be measured
directly from $1$. The fourth power in the $W^{1,4}$ term arises
naturally from the energy identity, but we do not claim that this
exponent is optimal.
\end{remark}

\section*{Declaration of generative AI and AI-assisted technologies in the manuscript preparation process}
The author used ChatGPT (OpenAI) primarily for literature searches and
English language editing. In particular, it helped identify three key
references used in the proofs. The work of Shi and
Tam~\cite{ShiTam02} is used in Theorem~\ref{main:boundary}.
The proof of Theorem~\ref{stab:global} uses the stability results of
Engelstein, Neumayer, and Spolaor~\cite{ENS22}, together with the
energy identity of Case~\cite{Case21}.
The author independently developed the reference curvature arguments for
Theorems~\ref{main:boundary}, \ref{main:constant}, and~\ref{main:quotient}.
The author reviewed, edited, and verified the content and takes full
responsibility for the manuscript.

\enlargethispage{3\baselineskip}
\end{document}